\documentclass[11pt,letterpaper,reqno]{amsart}

\usepackage{amsmath, amsfonts, amssymb, latexsym, accents, mathrsfs}
\usepackage{enumitem}
\usepackage{empheq}     
\usepackage{cancel}     

\usepackage[utf8]{inputenc}

\usepackage{tikz}
\usepackage{pgfplots}
\pgfplotsset{compat=1.18}
\usepackage{float}  

\makeatletter
\@addtoreset{figure}{section} 
\makeatother

\usepackage{xcolor}
\usepackage{ulem}     
\usepackage{comment}

\usepackage{url}
\usepackage[numbers,sort&compress]{natbib}
\usepackage[bookmarksnumbered,colorlinks, linkcolor=blue, citecolor=red, 
pagebackref, bookmarks, breaklinks]{hyperref}
\usepackage[hyperpageref]{backref}

\newtheorem{Theorem}{Theorem}[section]

\newtheorem{Corollary}[Theorem]{Corollary}
\newtheorem{remark}[Theorem]{Remark}

\catcode`@=11 \@addtoreset{equation}{section} \catcode`@=12

\makeatletter
\def\blfootnote{\xdef\@thefnmark{}\@footnotetext}
\makeatother

\begin{document}
	\title[Global Estimate and Stampacchia Maximum Principle]{On a global estimate and a Stampacchia--type maximum principle for Lane--Emden systems}
	
\author{Leandro G. Fernandes Jr.}
\thanks{Departamento de Matemática, Universidade Federal de Roraima, 
	Av. Cap. Ene Garcêz 2413, Bairro Aeroporto, Boa Vista, Roraima, 69310-000, Brazil. 
	Emails: \texttt{leandro.fernandes@ufrr.br}, \texttt{leandro.fernandes@ufscar.br}}

\author{Edir J. F. Leite$^*$}
\thanks{Departamento de Matemática, Universidade Federal de São Carlos, 
	São Carlos, SP, 13565-905, Brazil. Email: \texttt{edirleite@ufscar.br}}

\thanks{$^*$Corresponding author.}
	
	\begin{abstract}
		
	We establish a global boundedness result for Lane--Emden systems involving general second-order elliptic operators in divergence form and arbitrary positive exponents whose product equals one. Furthermore, we observe that, for this class of systems --- and for certain operators in divergence form, including the case when both operators are the Laplacian --- it is not possible to recover the classical Stampacchia maximum principle as a particular case corresponding to single equations.

	\end{abstract}
	


	\maketitle
	
	\section{Introduction}\label{intro}
	
	In this paper, we investigate Stampacchia--type formulations and global boundedness results for Lane--Emden systems. The main aim is to advance the understanding of such systems by extending analytical techniques and classical results from the scalar case, corresponding to a single equation, to the more general framework of nonlinear coupled systems.

	To this purpose, we consider the following nonhomogeneous system of partial differential equations:
	\begin{align}\label{system}
		\begin{cases} 
			-\mathcal{L}_{1}u = \lambda\rho(x)|v|^{\alpha-1}v + f_{1}(x) & \text{in } \Omega, \\
			-\mathcal{L}_{2}v = \mu\tau(x)|u|^{\beta-1}u + f_{2}(x) & \text{in } \Omega,
		\end{cases}
	\end{align}
where $\Omega$ is an open subset of $\mathbb{R}^n$, \(n\geq 2\), $\alpha$ and $\beta$ are positive numbers, and ${\mathcal{L}}_t$ ($t = 1, 2$) denotes a second-order uniformly elliptic operator in general divergence form, given by
\[
{\mathcal{L}}_t = \partial_i\left(a_t^{ij}(x) \partial_j + b_t^i(x)\right) + c_t^i(x)\partial_i + d_t(x).
\]
The coefficients \( a_t^{ij} \), \( b_t^i \), \( c_t^i \), and \( d_t \) are measurable functions satisfying the standard ellipticity and boundedness conditions:\begin{equation}\label{H}
	a_t^{ij}(x) \xi_i \xi_j \geq e_t |\xi|^2, \quad 
	e_t^{-2} \sum \left( |b_t^i(x)|^2 + |c_t^i(x)| \right) + e_t^{-1} |d_t(x)| \leq \nu_t^2,
\end{equation}
for all \( \xi \in \mathbb{R}^2 \), \( x \in \Omega \), and for some constants \( e_t > 0 \), \( \nu_t \geq 0 \). In addition, we assume that \( a_t^{ij} \in C(\overline{\Omega}) \) and, without loss of generality, that \( a_t^{ij} = a_t^{ji} \) for all \( i, j, t = 1, 2 \).

Furthermore, the zero-order term of the operator is assumed to be nonpositive, in the sense that
\[
\int_{\Omega} \left( d_{t} v_{t} - b_{t}^{i} D_{i} v_{t} \right) dx \leq 0, \quad 
\forall\, v_{t} \in C_{0}^{1}(\Omega),\ v_{t} \geq 0,
\]
which replaces the standard assumption that the zero-order term is nonnegative.

	The functions \(\rho\), \(\tau\), \(f_1\), and \(f_2\) are assumed to belong to \(L^{q/2}(\Omega)\) for some \(q > n\). For our purposes in Section~2, we take the parameters \(\lambda\) and \(\mu\) to be real numbers, and the exponents \(\alpha\) and \(\beta\) to be positive real numbers satisfying the reciprocal relation \(\alpha\beta = 1\).

	Several results have been established for Lane--Emden systems involving uniformly elliptic linear operators of second order in non-divergence form, including the existence of principal eigenvalues, maximum principles and comparison results, ABP--type estimates, Harnack inequalities, among others (see, for example, \cite{MR4072483, MR4732371, MR1765542}). However, relatively little is known for Lane--Emden systems governed by operators in divergence form. We refer to \cite{EM3} for some weak local boundedness and Harnack--type estimates for such systems in the plane. A possible explanation for this gap lies in the fact that the regularity theory for non-divergence form operators differs substantially in nature from that of divergence-form operators. For a comprehensive treatment of the regularity theory for weakly coupled cooperative elliptic systems involving operators in divergence form, we refer the reader to the book by Giaquinta~\cite{Giaquinta-multiple}. A detailed discussion on the different regularity requirements for elliptic equations in divergence and non-divergence form can be found in~\cite{MR657523,MR161019}. For further developments concerning elliptic equations in divergence form, see also~\cite{MR2601069,MR89986,MR170091,MR1214781,MR226198,MR170096}.

   In the homogeneous case of system~\eqref{system}, with \(\mathcal{L}_1 = \mathcal{L}_2 = \Delta\) and \(\lambda\rho = \mu\tau = 1\), it was proved by Kamburov and Sirakov in \cite{MR4531770} that, for two-dimensional smooth bounded domains, positive solutions of the superlinear Lane--Emden system are bounded independently of the exponents, provided they are comparable. Furthermore, for dimensions \(n \geq 3\), uniform a priori estimates were established in \cite{MR4728293} by Dai and Wu for positive solutions of higher-order superlinear Lane--Emden systems under Navier boundary conditions, in bounded domains.

	The main purpose of this work is to address the possibility of extending two classical results from the scalar setting to Lane--Emden systems involving operators in divergence form. It is important to emphasize that the coupling and nonlinearity inherent to the system introduce significant difficulties in generalizing results that are well understood in the case of single equations. Namely:
	
	The first of these is Stampacchia's maximum principle:
	
	Let $u \in W^{1,2}(\Omega)$ be a function such that $\mathcal{L}_1 u \geq 0$ in $\Omega$. Then,
	\[
	\sup_{\Omega} u \leq \sup_{\partial\Omega} u^+.
	\]
	
	The second is a global boundedness result:
	
	Let \( \mathcal{L}_1 \) be an operator satisfying the conditions in \eqref{H}, and assume that $f \in L^{q/2}(\Omega)$ for some $q > n$. Let $u \in W^{1,2}(\Omega)$ be a subsolution of the problem $\mathcal{L}_1 u = f$ in $\Omega$, satisfying $u \leq 0$ on $\partial\Omega$. Then,
	\[
		\sup_{\Omega} u \leq C \left( \|u^{+}\|_{L^{2}(\Omega)} + k \right),
	\]
	for a constant $C = C(n, \nu_1, q, |\Omega|)$ and $k := e_1^{-1} \|f\|_{L^{q/2}(\Omega)}$.
	
	Concerning the first result, one may ask whether the weak maximum principle extends to systems. For instance, suppose that $(u,v) \in (W^{1,2}(\Omega))^2$ is a weak subsolution (see Section~\ref{section 2} for the definition of a weak subsolution) to the system
	\[
	\begin{cases}
		-\Delta u = \lambda \rho(x) |v|^{\alpha - 1}v & \text{in } \Omega, \\
		-\Delta v = \mu \tau(x) |u|^{\beta - 1}u & \text{in } \Omega, \\
	\end{cases}
	\]
	for some $(\lambda, \mu) \in \mathbb{R}^2 \setminus \{(0,0)\}$ and exponents $\alpha, \beta > 0$ satisfying $\alpha  \beta = 1$. Does it follow that
	\[
	\sup_{\Omega} u \leq \sup_{\partial\Omega} u^{+} \quad \mbox{and} \quad \sup_{\Omega} v \leq \sup_{\partial\Omega} v^{+}?
	\]
	We show that this is not true in general. In fact,  we will study cases in which, for every $(\lambda, \mu) \in \mathbb{R}^2 \setminus \{(0,0)\}$, there exists a weak subsolution $(u, v) \in (W^{1,2}(\Omega))^2$ to the system above such that at least one of the following inequalities
	\begin{equation}\label{PM}
		\sup_{\Omega} u \leq \sup_{\partial\Omega} u^{+}, \qquad  \sup_{\Omega} v \leq \sup_{\partial\Omega} v^{+}
	\end{equation}
	fails.

A more precise discussion regarding the possibility that both inequalities may fail will be given in Section~\ref{section 2}.

	Although the weak maximum principle in the sense of Stampacchia does not hold for such systems, we are still able to establish a global boundedness result. By employing the De Giorgi--Nash--Moser iteration technique in a suitably adapted form for the two-dimensional setting, we obtain uniform bounds for solutions corresponding to the full range of exponents $\alpha, \beta > 0$ satisfying $\alpha\beta = 1$. More precisely, our main result is the following:

\begin{Theorem} \label{global boundedness}
	Let \( \Omega \subset \mathbb{R}^2 \) be an open set, and let \( \mathcal{L}_t \) (for \( t = 1, 2 \)) be an operator satisfying the conditions in \eqref{H}. Assume that \( \rho, \tau, f_1, f_2 \in L^{q/2}(\Omega) \) for some \( q > 2 \), and that \( \lambda, \mu \in \mathbb{R} \). Let also \( \alpha, \beta > 0 \) be such that \( \alpha \beta = 1 \), and assume that \( \alpha \leq 1 \).
	
	Let \( (u,v) \in (W^{1,2}(\Omega))^2 \) be a subsolution of the Lane--Emden system \eqref{system}, satisfying \( u \leq 0 \) and \( v \leq 0 \) on \( \partial\Omega \). Then,
	\[
		\| u^{+} \|_{L^{\infty}(\Omega)} + \left( \| v^{+} \|_{L^{\infty}(\Omega)} \right)^{\alpha}
		\leq C \left[ \| u^{+} \|_{L^{2}(\Omega)} + \| v^{+} \|_{L^{2}(\Omega)}^{\alpha} + k_1 + k_2^\alpha \right],
	\]
	for a constant \( C = C(\nu_t, \lambda, \mu, \| \rho \|_{L^{q/2}(\Omega)}, \| \tau \|_{L^{q/2}(\Omega)}, \alpha, q, |\Omega|) \), and constants \( k_t \) given by
	\[
	k_t := e_t^{-1} \| f_t \|_{L^{q/2}(\Omega)}.
	\]
\end{Theorem}

\begin{remark}\label{remark}
	The argument used in the proof of Theorem~\ref{global boundedness} can be extended to the case \( q > n > 2 \), provided that
	\[
	\frac{q(n-2)}{n(q-2)} \leq \alpha, \beta \leq \frac{n(q-2)}{q(n-2)}.
	\]
	In particular, no restriction is required when \( \alpha = \beta = 1 \).
\end{remark}

We observe that the restriction in the case \( n > 2 \) stems from the Sobolev embedding, which limits the admissible range of \( \alpha \) and \( \beta \).

	When $f_1 \equiv 0$ and $f_2 \equiv 0$, the behavior of system~\eqref{system} becomes resonant when $\alpha\beta = 1$. This resonance motivates the choice of such exponents and, in particular, allows for a natural extension of results previously established in the scalar setting.

	As a direct consequence of the above global boundedness result, we derive a counterpart for weak solutions of biharmonic equations of the type
	\begin{equation}\label{bih}
		(-\Delta)^2 v = \lambda\rho(x)v + f(x) \quad \text{in } \Omega,
	\end{equation}
	defined on open sets $\Omega \subset \mathbb{R}^n$, which correspond to systems of the form
	\begin{equation}\label{S1.4}
		\begin{cases}
			-\Delta u = \lambda\rho(x) v + f(x) & \text{in } \Omega, \\
			-\Delta v = u & \text{in } \Omega.
		\end{cases}
	\end{equation}

	\textit{A priori} estimates for biharmonic functions (i.e., when \(f \equiv 0\) and \(\lambda = 0\)) and, more generally, for polyharmonic functions, constitute classical and well-studied problems. For instance, in \cite{MR2086750}, Busca and Sirakov established a Harnack inequality for classical solutions of the polyharmonic equation
	\[
	(-\Delta)^m v = \lambda\rho(x)v + f(x) \quad \text{in } \Omega \subset \mathbb{R}^n,
	\]
	where \(\rho \in L^{\infty}(\Omega)\). Subsequently, Caristi and Mitidieri~\cite{MR2240050} proved a Harnack inequality for weak solutions of the biharmonic equation (i.e., \(m = 2\)) in the Sobolev space \(W^{2,2}(\Omega)\), under the assumption \(f \equiv 0\), and where the potential \(\rho\) belongs to \(K^{n,2}(\Omega)\), the natural Kato class associated with the operator \((-\Delta)^2\) (see Definition~2.2 in~\cite{MR2240050}).

	Theorem~\ref{global boundedness} and Remark~\ref{remark}, when applied to the system~\eqref{S1.4}, yields the following estimate for weak subsolutions of nonhomogeneous biharmonic equations:

\begin{Corollary}
	Let $\Omega \subset \mathbb{R}^n$ be an open set, and assume that $\rho, f \in L^{q/2}(\Omega)$ for some \( q > n \), and let \( \lambda \in \mathbb{R} \).
	
	Let \( v \in W^{3,2}(\Omega) \) be a weak subsolution of the biharmonic equation~\eqref{bih}, satisfying \( v \leq 0 \) and \( -\Delta v \leq 0 \) on \( \partial\Omega \). Then,
	\begin{align*}
		\Vert (-\Delta v)^+ \Vert_{L^{\infty}(\Omega)} + \left(\Vert v^+ \Vert_{L^{\infty}(\Omega)}\right)^{\alpha}
		\leq\ &C \left( \Vert (-\Delta v)^+ \Vert_{L^2(\Omega)} + \Vert v^+ \Vert_{L^2(\Omega)} \right. \\
		&\left. + \Vert f \Vert_{L^{q/2}(\Omega)} \right),
	\end{align*}
	for some constant \( C = C\left(\lambda, n, \Vert \rho \Vert_{L^{q/2}(\Omega)}, q, \vert\Omega\vert\right) \).
\end{Corollary}

In this context, the regularity \( v \in W^{3,2}(\Omega) \) is essential, as it ensures that \( (u, v) \in (W^{1,2}(\Omega))^2 \), where \( u := -\Delta v \).

Observe that by taking \( u = 1 - |x|^2 \), one verifies that \( (-\Delta)^2 u = 0 \) in \( B_1(0) \). This simple example illustrates that the weak maximum principle, in the sense of Stampacchia, does not hold for biharmonic equations.

Finally, at the end of Subsection~2.1, after presenting examples directly related to the failure of the Stampacchia--type maximum principle for Lane--Emden systems in which both operators are taken to be the Laplacian, we include a brief discussion on the possible validity of ABP--type estimates, maximum and comparison principles, and the existence of eigenvalue curves in the setting of divergence-form operators, among other issues. We also address the inherent limitations on which operators can be simultaneously written in both divergence and non-divergence forms, and emphasize that, in such cases, the appropriate analytical framework is the one developed for the non-divergence setting, as established in~\cite{MR1765542, MR4072483}.\\

	\section{ Proofs of the Main Results }\label{section 2}

	We begin this section with the following definition. A pair \((u, v)\) of weakly differentiable functions is called a \textit{weak subsolution} of system~\eqref{system} in \(\Omega\) if the functions \(A_{t}^{i}(x, u, Du)\) and \(B_{t}(x, u, Du)\), for \(t = 1, 2\), are locally integrable, and the following inequalities hold for all nonnegative test functions \(w_1, w_2 \in C_0^1(\Omega)\):
	\begin{align}\label{subsolution-1}
		\int_{\Omega} \left[ A_{1}^{i}(x, u, Du) D_{i}w_{1} - B_{1}(x, u, Du) w_{1} \right]\, dx 
		\leq \int_{\Omega} \lambda \rho(x) |v|^{\alpha-1} v w_{1} \, dx,
	\end{align}
	\begin{align}\label{subsolution-2}
		\int_{\Omega} \left[ A_{2}^{i}(x, v, Dv) D_{i}w_{2} - B_{2}(x, v, Dv) w_{2} \right]\, dx 
		\leq \int_{\Omega} \mu \tau(x) |u|^{\beta-1} u w_{2} \, dx,
	\end{align}
	where the functions \(A_t^i, B_t : \Omega \times \mathbb{R} \times \mathbb{R}^2 \rightarrow \mathbb{R}\), for \(i,t = 1,2\), are defined by
	\begin{align}\label{conditions-0}
		\begin{cases}
			A_{t}^{i}(x, z_t, p) = a_{t}^{ij}(x) p_{j} + b_{t}^{i}(x) z_t, \\[4pt]
			B_{t}(x, z_t, p) = c_{t}^{i}(x) p_{i} + d_{t}(x) z_t + f_{t}(x),
		\end{cases}
	\end{align}
	with \(b_t = (b_t^1, b_t^2)\) and \(c_t = (c_t^1, c_t^2)\) for \(t = 1,2\).
	
	It is worth noting that the test functions \(w_1\) and \(w_2\), initially taken in \(C_0^1(\Omega)\), can be extended by density to arbitrary nonnegative functions in \(W_0^{1,2}(\Omega)\).

	A straightforward computation (see Chapter~8 in~\cite{Trudingerbook})
	 shows that, for any \(0 < \varepsilon \leq 1\), one has
	\begin{align}\label{conditions-1}
		\begin{cases}
			p_{i} A_{t}^{i}(x, z_t, p) 
			\geq \dfrac{e_{t}}{2} \left( |p|^2 - 2 \overline{b}_{t} \, \overline{z}_t^{2} \right), \\[6pt]
			|\overline{z}_t B_{t}(x, z_t, p)|
			\leq \dfrac{e_{t}}{2} \left( \varepsilon |p|^2 + \dfrac{\overline{b}_{t}}{\varepsilon} \, \overline{z}_t^{2} \right),
		\end{cases}
	\end{align}
where \( \overline{z}_t = |z_t| + k_t \), with fixed constants \( k_t > 0 \), and
\[
\overline{b}_t = e_t^{-2} \left( |b_t|^2 + |c_t|^2 \right) + e_t^{-1} \left( |d_t| + k_t^{-1} |f_t| \right), \quad \mbox{with } k_t = e_t^{-1} \|f_t\|_{L^{q/2}(\Omega)}.
\]

	\subsection{Examples and Discussions}
	
We begin this section by noting that not every subsolution \((u,v) \in (W^{1,2}(\Omega))^2\) of the Lane--Emden system \eqref{system}, with \(f_1 \equiv f_2 \equiv 0\), satisfies the pointwise estimate \eqref{PM}, even in the simplest case where \(\mathcal{L}_1 = \mathcal{L}_2 = \Delta\).

What we can guarantee through Stampacchia’s maximum principle for a single equation is that, if \((u,v) \in (W^{1,2}(\Omega))^2\) is a subsolution of system~\eqref{system}, with \(f_1 \equiv f_2 \equiv 0\), and \(\lambda v \leq 0\), \(\mu u \leq 0\) in \(\Omega\), then both inequalities in~\eqref{PM} hold.

	\begin{figure}[H]
			\centering
			\begin{tikzpicture}
				\begin{axis}[
					axis lines=middle,
					xmin=0, xmax=5.2,
					ymin=0, ymax=10.5,
					xtick=\empty,
					ytick=\empty,
					width=10cm,
					height=8cm,
					clip=false,
					axis line style={->}
					]
					
					\addplot[
					domain=0.1:5,
					samples=300,
					thick,
					black
					]{1/x};
					
					\addplot[
					domain=0:5,
					samples=2,
					dashed,
					thick,
					black
					]{1.2*x};
					
					\addplot[
					only marks,
					mark=*,
					mark size=2pt,
					black
					] coordinates {(0.91287, 1.09545)};
					
					\draw[->, thick] (axis cs:1.1,2.3) -- (axis cs:0.954, 1.26);
					\node at (axis cs:1.2,2.7) {$(\lambda_1(a), \mu_1(a))$};
					
					\draw[->, thick] (axis cs:4.0,6.9) -- (axis cs:3.0,3.7);
					\node at (axis cs:4.0,7.2) {$\mu = a  \lambda$};
					
					\draw[->, thick] (axis cs:3.4,1.3) -- (axis cs:2.5,0.5);
					\node at (axis cs:3.6,1.3) {$\Lambda_1$};
					
				\end{axis}
				
				\node at (8.1,-0.25) {\large $\lambda$};
				\node at (-0.25,5.6) {\large $\mu$};
			\end{tikzpicture} \label{fig:eigenvalue-curve}
			\caption{Graph of the eigenvalue curve associated with the homogeneous Lane--Emden system involving the Laplace operator, together with the line $\mu = a  \lambda$, which intersects the curve at the point $(\lambda_1(a), \mu_1(a))$.}
			
		\end{figure}
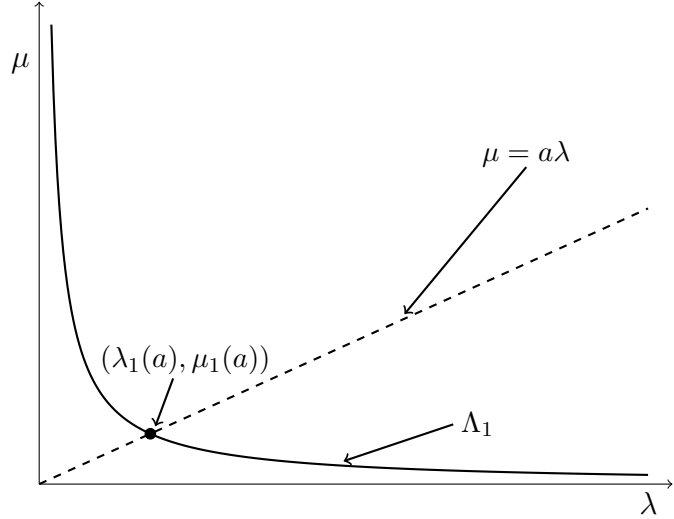
		
To address situations other, we now invoke the result established by Montenegro in \cite{MR1765542} for operators in non-divergence form, specifically in the case where both operators are the Laplacian, which applies to both the divergence and non-divergence formulations. In Figure~\the\value{figure}, we display the graph of the first eigenvalue curve.
	
To begin illustrating the example, let us consider a fixed pair \((\lambda, \mu) \in \mathbb{R}^2\) lying outside the set \(\overline{\mathcal{C}_1} \setminus \Lambda_1\), where \(\Lambda_1\) denotes the trace of the curve of all principal eigenvalues associated with the homogeneous Lane--Emden system in which both operators are the Laplacian, and \(\mathcal{C}_1\) is the open region in the first quadrant lying below \(\Lambda_1\). For more details on this curve, see \cite{MR1765542, MR4072483}. Suppose further that \((\lambda, \mu) \in \mathbb{R}_+^2\); then it holds that \(\lambda \geq \lambda_1(a)\) and \(\mu \geq \mu_1(a)\), where \((\lambda_1(a), \mu_1(a))\) denotes the principal eigenvalue associated with the problem
\begin{equation*}
	\left\{
	\begin{array}{ll}
		-\Delta u = \lambda\rho(x)\vert v\vert^{\alpha-1}v  & \text{in } \Omega,\\[4pt]
		-\Delta v = \mu\tau(x)\vert u\vert^{\beta-1}u  & \text{in } \Omega,\\[4pt]
		u = v = 0 & \text{on } \partial\Omega,
	\end{array}
	\right.
\end{equation*}
with \(a = \mu / \lambda\), see Figure~\the\value{figure}. The corresponding pair of eigenfunctions \((\varphi_a, \psi_a)\) belongs to \((W^{2,n}(\Omega))^2\), the weight functions \(\rho\) and \(\tau\) belong to \(L^n(\Omega)\), and the exponents fulfill the relation \(\alpha  \beta = 1\); see \cite{MR1765542}. Then \((\varphi_a, \psi_a)\) is a positive weak subsolution of
	\[
	\left\{
	\begin{array}{llll}
		-\Delta \varphi_a &= \lambda\rho(x) \psi_a^{\alpha} & \text{in } \Omega,\\ 
		-\Delta \psi_a &= \mu\tau(x) \varphi_a^{\beta} & \text{in } \Omega,
	\end{array}
	\right.
	\]
	with \(\varphi_a = \psi_a = 0\) on \(\partial\Omega\). Neither of the estimates in \eqref{PM} holds in this case.

	Now suppose that \(\lambda < 0\). By Propositions~2.10 and 2.11 in \cite{MR1765542}, there exists a sufficiently large constant \(a > 0\) such that \(\lambda \leq -\lambda_1(a)\) and \(\mu \geq -\mu_1(a)\). In this case, we obtain the following system:
	\[
	\left\{
	\begin{array}{llll}
		-\Delta \varphi_a &\leq \lambda\rho(x) \, |\psi_a|^{\alpha-1}(-\psi_a)  & \text{in } \Omega,\\[4pt]
		-\Delta (-\psi_a) &\leq \mu\tau(x)\, \varphi_a^{\beta} & \text{in } \Omega,
	\end{array}
	\right.
	\]
	with boundary conditions \(\varphi_a = -\psi_a = 0\) on \(\partial\Omega\). In this case, at least one of the estimates in~\eqref{PM} does not hold.

	In the case where \(\lambda \geq 0\) and \(\mu < 0\), we again invoke Propositions~2.10 and 2.11 from \cite{MR1765542}. Accordingly, we can choose a sufficiently small constant \(a > 0\) such that \(\lambda \geq -\lambda_1(a)\) and \(\mu \leq -\mu_1(a)\).
	
	As a consequence, we obtain the following system:
	\[
	\left\{
	\begin{array}{llll}
		-\Delta (-\varphi_a) &\leq \lambda\rho(x)\, \psi_a^{\alpha} & \text{in } \Omega,\\[4pt]
		-\Delta \psi_a &\leq \mu\tau(x)\, |\varphi_a|^{\beta - 1}(-\varphi_a) & \text{in } \Omega,
	\end{array}
	\right.
	\]
	with boundary conditions \(-\varphi_a = \psi_a = 0\) on \(\partial\Omega\). Thus, once again, at least one of the inequalities in~\eqref{PM} fails to hold.
	
	It is worth mentioning that the examples above can be extended to the class of non-divergence form operators that can also be written in divergence form. This extension is based on the eigenvalue curves established by Montenegro in~\cite{MR1765542}. For a broader discussion (see also p.\,2 in~\cite{MR4106729}) on the intersection between these two classes of operators.

We now present another class of examples, valid for any \(\lambda \geq 0\) and \(\mu > 0\). Namely, observe that the functions \(v(x) = \mu(1 - |x|^2)\) and \(u(x) = 2n\) define a classical solution to the system
\[
\left\{
\begin{array}{ll}
	-\Delta u = 0 \leq \lambda \rho(x)\, v^{\alpha} & \text{in } B_1(0),\\[4pt]
	-\Delta v = 2\mu n \leq \mu \tau(x)\, u^{\beta} & \text{in } B_1(0),
\end{array}
\right.
\]
with \(\alpha\beta = 1\), \(\tau \geq 1\), and \(\beta \geq 1\). In this case, we have
	\[
	\sup_{\Omega} v = \mu > \sup_{\partial\Omega} v^+ = 0,
	\]
	showing that at least one of the estimates in~\eqref{PM} fails to hold.
	
	We can also construct a similar example where the inequality for \(u\) fails. Let \(\lambda > 0\), \(\mu \geq 0\), and consider the functions \(u(x) = \lambda(1 - |x|^2)\) and \(v(x) = 2n\). Then \((u,v)\) is a classical solution to the system
	\[
	\left\{
	\begin{array}{ll}
		-\Delta u = 2\lambda n \leq \lambda \rho(x)\, v^{\alpha} & \text{in } B_1(0),\\[4pt]
		-\Delta v = 0 \leq \mu \tau(x)\, u^{\beta} & \text{in } B_1(0),
	\end{array}
	\right.
	\]
	with \(\alpha  \beta = 1\), \(\rho \geq 1\), and \(\alpha \geq 1\). In this case, it follows that
	\[
	\sup_{\Omega} u = \lambda > \sup_{\partial\Omega} u^+ = 0,
	\]
	so once again, at least one of the inequalities in~\eqref{PM} does not hold.

	In the particular examples considered, we observe that both inequalities fail to hold in the portion of the first quadrant lying outside \(\mathcal{C}_1\), that is, on and above \(\Lambda_1\). In the remaining cases, only one of the inequalities fails to satisfy estimate~\eqref{PM}.

	Based on the preceding discussion, it is natural to pose the following questions: Does there exist a mixed inequality of Stampacchia type involving both the supremum and the infimum of the subsolutions? Moreover, is it possible to establish an ABP--type estimate in which a constant appears alongside the supremum of the subsolutions at the boundary, depending explicitly on the dimension, the functions \(\rho\) and \(\tau\), the parameters \(\lambda\) and \(\mu\), the domain \(\Omega\), the exponent \(\beta\), and other relevant quantities that do not involve the subsolutions themselves? All these questions concern the Lane--Emden systems in divergence form.

	Although the theories of second-order linear operators in divergence and non-divergence form exhibit distinct characteristics, there exists a notable intersection that must be taken into account. This observation suggests that, if it is possible to establish a general maximum principle for the Lane-–Emden system, then such a principle should hold only inside the region $\overline{\mathcal{C}_1} \setminus \Lambda_1$, where $\mathcal{C}_1$ and $\Lambda_1$ are defined analogously to those above, with respect to the corresponding operators; see \cite{MR4072483} and \cite{MR1765542} for further details.

	Thus, another open question is whether a similar structure could exist for divergence form operators, including, for instance, an eigenvalue curve, a region where the maximum principle holds, and possibly a comparison principle, as in the non-divergence case.

	We are now in a position to state the main result of this paper concerning the global boundedness of subsolutions.
	
	\subsection{ Global Boundedness Estimate}

	Throughout this section, we set \(q^{\ast} = \dfrac{2q}{q - 2}\) for all \(q > 2\).
	 \begin{proof}[Proof of Theorem~\ref{global boundedness}] For $\beta_{1}\geq 1$ and $N_t>k_t $, let us define $H_t\in C^{1}([k_t,+\infty)$ by setting $H_t(z)=z^{\beta_{1}}-k_t^{\beta_{1}}$ for $z\in [k_t,N_t]$ and $H_t(z)=\beta_{1}N_t^{\beta_{1}-1}\left(z-N_{t}\right)+\left(N_{t}^{\beta_{1}}-k_{t}^{\beta_{1}}\right)$ for $z\geq N_t$.
		
		Let us now set $u_{1}=u^{+}+k_1$ and $u_{2}=v^{+}+k_2$  and take \begin{align*}
			G_t(u_{t})=\int_{k_t}^{u_{t}}\vert H_t'(s)\vert^{2}ds.
		\end{align*}
		
		We will now divide the proof into steps.
		
		\noindent \textsl{\underline{Step one}}: There exist constants \( C_1 = C_1(s_1, \Omega) > 0 \) and \( C_2 = C_2(s_2, \Omega) > 0 \) such that

		 \begin{align*}
			&\Vert H_{1}(u_{1}) \Vert_{L^{2\chi}(\Omega)}^{2} 
			\leq C_1 \Vert\overline{b}_1\Vert_{L^{q/2}(\Omega)}\Vert u_{1} H_{1}'(u_{1}) \Vert_{L^{q^{\ast}}(\Omega)}^{2} 
			\\
			&+ |\lambda| \int_{\Omega} |\rho(x)|\, |v|^{\alpha} G_{1}(u_{1}) \, dx
		\end{align*}
		and
		\begin{align*}
			&\Vert H_{2}(u_{2}) \Vert_{L^{2\eta}(\Omega)}^{2} 
			\leq C_2\Vert\overline{b}_2\Vert_{L^{q/2}(\Omega)}\Vert u_{2} H_{2}'(u_{2}) \Vert_{L^{q^{\ast}}(\Omega)}^{2}\\ 
			&+ |\mu| \int_{\Omega} |\tau(x)|\, |u|^{\beta} G_{2}(u_{2}) \, dx,
		\end{align*}
		
		where $\chi=\frac{s_{1}}{s_{1}-2}$ with $2<s_{1}<+\infty$, and $\eta=\frac{s_{2}}{s_{2}-2}$ with $2<s_{2}<+\infty$.
		
	Indeed, since the pair \((u, v)\) is a weak subsolution of system~\eqref{system} — more precisely, it satisfies inequalities~\eqref{subsolution-1} and~\eqref{subsolution-2} — we may invoke the identities~\eqref{conditions-0}, the structural conditions~\eqref{conditions-1}, the inequality \( G_t(s) \leq s G_t'(s) \), and the equalities
\[
Du = Du_1 \quad \text{and} \quad Dv = Du_2 \quad \text{a.e. in } \Omega,
\]
to derive the following estimates:

	\begin{align*}
		\int_{\Omega} |Du_1|^2 G_1'(u_1) \, dx 
		&\leq \int_{\Omega} \left( \overline{b}_1 G_1'(u_1) u_1^2 + \frac{2}{e_1} G_1(u_1) |B_1(x, u, Du)| \right) dx \\
		&\leq \varepsilon_3 \int_{\Omega} G_1'(u_1) |Du_1|^2 \, dx 
		+ \left(1 + \frac{1}{\varepsilon_3} \right) \int_{\Omega} \overline{b}_1 G_1'(u_1) u_1^2 \, dx \\
		&\quad + \int_{\Omega} |\lambda \rho(x)|\, |v|^{\alpha} G_1(u_1) \, dx,
	\end{align*}
	and similarly,
	\begin{align*}
		\int_{\Omega} |Du_2|^2 G_2'(u_2) \, dx 
		&\leq \int_{\Omega} \left( \overline{b}_2 G_2'(u_2) u_2^2 + \frac{2}{e_2} G_2(u_2) |B_2(x, v, Dv)| \right) dx \\
		&\leq \varepsilon_4 \int_{\Omega} G_2'(u_2) |Du_2|^2 \, dx 
		+ \left(1 + \frac{1}{\varepsilon_4} \right) \int_{\Omega} \overline{b}_2 G_2'(u_2) u_2^2 \, dx \\
		&\quad + \int_{\Omega} |\mu \tau(x)|\, |u|^{\beta} G_2(u_2) \, dx,
	\end{align*}
where \( \varepsilon_3 \) and \( \varepsilon_4 \) are arbitrary real numbers in the interval \( (0,1) \). Here, we employed the fact that \( G_1(u_1) \) and \( G_2(u_2) \) are valid test functions in the weak formulations corresponding to \eqref{subsolution-1} and \eqref{subsolution-2}, respectively.

		Now, choosing \( \varepsilon_3 = \varepsilon_4 = \tfrac{1}{2} \), we obtain
		
		\begin{align*}
			\int_{\Omega} \vert Du_{1} \vert^{2} G_{1}'(u_{1}) \, dx &\leq 6 \int_{\Omega} \overline{b}_{1} G_{1}'(u_{1}) u_{1}^{2} \, dx \\
			&\quad + \int_{\Omega} \vert\lambda \rho(x)\vert \vert v \vert^{\alpha} G_{1}(u_{1}) \, dx,
		\end{align*}
		
		and
		
		\begin{align*}
			\int_{\Omega} \vert Du_{2} \vert^{2} G_{2}'(u_{2}) \, dx &\leq 6 \int_{\Omega} \overline{b}_{2} G_{2}'(u_{2}) u_{2}^{2} \, dx \\
			&\quad + \int_{\Omega} \vert\mu \tau(x)\vert \vert u \vert^{\beta} G_{2}(u_{2}) \, dx.
		\end{align*}
		
		Since \( H_{1}(u_{1}) \) and \( H_{2}(u_{2}) \) belong to \( W^{1,2}_{0}(\Omega) \), we can apply the Sobolev inequality and Hölder's inequality to obtain		
		
	\begin{align*}
	&\Vert H_{1}(u_{1}) \Vert_{L^{2\chi}(\Omega)}^{2} 
	\leq C_1 \Vert\overline{b}_1\Vert_{L^{q/2}(\Omega)}\Vert u_{1} H_{1}'(u_{1}) \Vert_{L^{q^{\ast}}(\Omega)}^{2} 
	\\
	&+ |\lambda| \int_{\Omega} |\rho(x)|\, |v|^{\alpha} G_{1}(u_{1}) \, dx
	\end{align*}
		and
	\begin{align*}
			&\Vert H_{2}(u_{2}) \Vert_{L^{2\eta}(\Omega)}^{2} 
			\leq C_2\Vert\overline{b}_2\Vert_{L^{q/2}(\Omega)}\Vert u_{2} H_{2}'(u_{2}) \Vert_{L^{q^{\ast}}(\Omega)}^{2}\\ 
			&+ |\mu| \int_{\Omega} |\tau(x)|\, |u|^{\beta} G_{2}(u_{2}) \, dx,
		\end{align*}
		
		which finishes the step one.

   \noindent
   \textsl{\underline{Step two}}: Let \( \varepsilon_j \in (0,1) \) be arbitrary for \( j = 5,6,7,8 \). Then, there exists a constant \( C > 0 \), depending only on \( |\Omega| \) and \( q \), such that
   \begin{align*}
   	&\int_{\Omega} |\lambda|\, |\rho(x)|\, |v|^{\alpha} G_{1}(u_{1}) \, dx 
   	\leq \beta_{1}^{2} \frac{2\beta_{1}-1}{2\beta_{1}}  
   	\Vert \lambda \rho \Vert_{L^{q/2}(\Omega)} 
   	\bigg(
   	\varepsilon_{5} \Vert u_{1}^{\beta_{1}} \Vert_{L^{2\chi}(\Omega)} 
   	+ {}\\
   	&\qquad\qquad\qquad\qquad\qquad\qquad+
   	C \varepsilon_{5}^{-\sigma_{1}} \Vert u_{1}^{\beta_{1}} \Vert_{L^{q^{\ast}}(\Omega)}
   	\bigg)^{2} \\
   	&\quad\qquad\qquad  +\frac{\beta_{1}^{2}}{2\beta_{1}} 
   	\Vert \lambda \rho \Vert_{L^{q/2}(\Omega)} 
   	\bigg(
   	\varepsilon_{6} \Vert u_{2}^{\alpha \beta_{1}} \Vert_{L^{2\chi}(\Omega)} 
   	+ C \varepsilon_{6}^{-\sigma_{1}} \Vert u_{2}^{\alpha \beta_{1}} \Vert_{L^{q^{\ast}}(\Omega)}
   	\bigg)^{2}
   \end{align*}
   and
   \begin{align*}
   	&\int_{\Omega} |\mu|\, |\tau(x)|\, |u|^{\beta} G_{2}(u_{2}) \, dx 
   	\leq \beta_{1}^{2} \frac{2\beta_{1}-1}{2\beta_{1}} 
   	\Vert \mu \tau \Vert_{L^{q/2}(\Omega)} 
   	\bigg(
   	\varepsilon_{7} \Vert u_{2}^{\beta_{1}} \Vert_{L^{2\eta}(\Omega)} 
   	 {}\\
   	&\quad\quad\quad\quad\qquad\qquad\qquad\quad\quad
   	+C \varepsilon_{7}^{-\sigma_{1}} \Vert u_{2}^{\beta_{1}} \Vert_{L^{q^{\ast}}(\Omega)}
   	\bigg)^{2} \\
   	&\quad\qquad\qquad  + \frac{\beta_{1}^{2}}{2\beta_{1}} 
   	\Vert \mu \tau \Vert_{L^{q/2}(\Omega)} 
   	\bigg(
   	\varepsilon_{8} \Vert u_{1}^{\beta_{1}} \Vert_{L^{2\eta}(\Omega)} 
   +C \varepsilon_{8}^{-\sigma_{2}} \Vert u_{1}^{\beta_{1}} \Vert_{L^{q^{\ast}}(\Omega)}
   	\bigg)^{2\beta},
   \end{align*}
   
   \noindent
  where \( \chi = \frac{s_1}{s_1 - 2} \), with \( 2 < s_1 < q \), and \( \eta = \frac{s_2}{s_2 - 2} \), with \( 2 < s_2 < \frac{2\beta q}{\beta q - q + 2} \). The exponents \( \sigma_1 \) and \( \sigma_2 \) are given respectively by \( \sigma_1 = \left( \frac{1}{2} - \frac{1}{q^*} \right) / \left( \frac{1}{q} - \frac{1}{2\chi} \right) \) and \( \sigma_2 = \left( \frac{1}{2} - \frac{1}{q^*} \right) / \left( \frac{1}{q} - \frac{1}{2\eta} \right) \). The quantity \( |\Omega| \) denotes the Lebesgue measure of the domain \( \Omega \subset \mathbb{R}^2 \).
    
		In fact, by first applying Young's inequality and then Hölder's inequality, we obtain
		\begin{align}\label{2.6}
			&\int_{\Omega} |\lambda \rho(x)|\, |v|^{\alpha} G_{1}(u_{1}) \, dx 
			\leq \int_{\Omega} |\lambda \rho(x)|\, u_{2}^{\alpha} \beta_{1}^{2} u_{1}^{2\beta_{1} - 1} \, dx \nonumber \\
			&\leq \beta_{1}^{2} \int_{\Omega} |\lambda \rho(x)|^{\frac{2\beta_{1} - 1}{2\beta_{1}}} u_{1}^{2\beta_{1} - 1} 
			|\lambda \rho(x)|^{\frac{1}{2\beta_{1}}} u_{2}^{\alpha} \, dx \nonumber \\
			&\leq \beta_{1}^{2}  \frac{2\beta_{1} - 1}{2\beta_{1}} \int_{\Omega} |\lambda \rho(x)|\, u_{1}^{2\beta_{1}} \, dx 
			+ \frac{\beta_{1}^{2}}{2\beta_{1}} \int_{\Omega} |\lambda \rho(x)|\, u_{2}^{2\beta_{1} \alpha} \, dx \nonumber \\
			&\leq \beta_{1}^{2} \frac{2\beta_{1} - 1}{2\beta_{1}} \Vert \lambda \rho \Vert_{L^{q/2}(\Omega)} 
			\left( \Vert u_{1} \Vert_{L^{\frac{2\beta_{1}q}{q-2}}(\Omega)} \right)^{2\beta_{1}} 
			\nonumber\\
			&+ \frac{\beta_{1}^{2}}{2\beta_{1}} \Vert \lambda \rho \Vert_{L^{q/2}(\Omega)} 
			\left( \Vert u_{2} \Vert_{L^{\frac{2\beta_{1} \alpha q}{q-2}}(\Omega)} \right)^{2\beta_{1} \alpha}.
		\end{align}
		
		Now, for \( 2 < s_{1} < q \) and \( \chi = \frac{s_{1}}{s_{1} - 2} \), it follows from~\eqref{2.6}, together with the interpolation inequality and Hölder's inequality, that
		\begin{align*}
			&|\lambda| \int_{\Omega} |\rho(x)|\, |v|^{\alpha} G_{1}(u_{1}) \, dx 
			\leq \beta_{1}^{2}  \frac{2\beta_{1} - 1}{2\beta_{1}} \,
			\Vert \lambda \rho \Vert_{L^{q/2}(\Omega)} \, \bigg(
			\varepsilon_{5} \Vert u_{1}^{\beta_{1}} \Vert_{L^{2\chi}(\Omega)} \nonumber \\
			&\quad + \varepsilon_{5}^{-\sigma_{1}} \vert\Omega\vert^{1/q^{\ast}-1/2} \Vert u_{1}^{\beta_{1}} \Vert_{L^{q^{\ast}}(\Omega)}
			\bigg)^{2} \nonumber \\
			&\quad + \frac{\beta_{1}^{2}}{2\beta_{1}} \,
			\Vert \lambda \rho \Vert_{L^{q/2}(\Omega)} \, \bigg(
			\varepsilon_{6} \Vert u_{2}^{\alpha \beta_{1}} \Vert_{L^{2\chi}(\Omega)} 
			+ \varepsilon_{6}^{-\sigma_{1}} \vert\Omega\vert^{1/q^{\ast}-1/2} \Vert u_{2}^{\alpha \beta_{1}} \Vert_{L^{q^{\ast}}(\Omega)}
			\bigg)^{2}.
		\end{align*}
		
		Based on similar arguments, we obtain
		
		\begin{align*}
			&\int_{\Omega} |\mu \tau(x)|\, |u|^{\beta} G_{2}(u_{2}) \, dx 
			\leq \int_{\Omega} |\mu \tau(x)|\, u_{1}^{\beta} \beta_{1}^{2} u_{2}^{2\beta_{1} - 1} \, dx \nonumber \\
			&\leq \beta_{1}^{2} \int_{\Omega} |\mu \tau(x)|^{\frac{2\beta_{1} - 1}{2\beta_{1}}} u_{2}^{2\beta_{1} - 1} 
			|\mu \tau(x)|^{\frac{1}{2\beta_{1}}} u_{1}^{\beta} \, dx \nonumber \\
			&\leq \beta_{1}^{2}  \frac{2\beta_{1} - 1}{2\beta_{1}} 
			\int_{\Omega} |\mu \tau(x)|\, u_{2}^{2\beta_{1}} \, dx 
			+ \frac{\beta_{1}^{2}}{2\beta_{1}} 
			\int_{\Omega} |\mu \tau(x)|\, u_{1}^{2\beta_{1}\beta} \, dx \nonumber \\
			&\leq \beta_{1}^{2} \frac{2\beta_{1} - 1}{2\beta_{1}} 
			\Vert \mu \tau \Vert_{L^{q/2}(\Omega)} 
			\left( \Vert u_{2} \Vert_{L^{\frac{2\beta_{1}q}{q-2}}(\Omega)} \right)^{2\beta_{1}} \nonumber\\ 
			&+ \frac{\beta_{1}^{2}}{2\beta_{1}} 
			\Vert \mu \tau \Vert_{L^{q/2}(\Omega)} 
			\left( \Vert u_{1} \Vert_{L^{\frac{2\beta_{1}\beta q}{q-2}}(\Omega)} \right)^{2\beta_{1}\beta} \nonumber \\
			&\leq \beta_{1}^{2} \frac{2\beta_{1} - 1}{2\beta_{1}} 
			\Vert \mu \tau \Vert_{L^{q/2}(\Omega)} 
			\left( \varepsilon_{7} \Vert u_{2}^{\beta_{1}} \Vert_{L^{2\eta}(\Omega)} 
			+ \varepsilon_{7}^{-\sigma_{2}} \Vert u_{2}^{\beta_{1}} \Vert_{L^{2}(\Omega)} \right)^{2} \nonumber \\
			&\quad + \frac{\beta_{1}^{2}}{2\beta_{1}} 
			\Vert \mu \tau \Vert_{L^{q/2}(\Omega)} 
			\left( \varepsilon_{8} \Vert u_{1}^{\beta_{1}} \Vert_{L^{2\eta}(\Omega)} 
			+ \varepsilon_{8}^{-\sigma_{2}} \Vert u_{1}^{\beta_{1}} \Vert_{L^{2}(\Omega)} \right)^{2\beta} \nonumber \\
			&\leq \beta_{1}^{2} \frac{2\beta_{1} - 1}{2\beta_{1}} 
			\Vert \mu \tau \Vert_{L^{q/2}(\Omega)} 
			\left( \varepsilon_{7} \Vert u_{2}^{\beta_{1}} \Vert_{L^{2\eta}(\Omega)} 
			+ \varepsilon_{7}^{-\sigma_{2}} \Omega\vert^{1/q^{\ast}-1/2} \Vert u_{2}^{\beta_{1}} \Vert_{L^{q^{\ast}}(\Omega)} \right)^{2} \nonumber \\
			&\quad + \frac{\beta_{1}^{2}}{2\beta_{1}} 
			\Vert \mu \tau \Vert_{L^{q/2}(\Omega)} 
			\left( \varepsilon_{8} \Vert u_{1}^{\beta_{1}} \Vert_{L^{2\eta}(\Omega)} 
			+ \varepsilon_{8}^{-\sigma_{2}} \vert \Omega\vert^{1/q^{\ast}-1/2} \Vert u_{1}^{\beta_{1}} \Vert_{L^{q^{\ast}}(\Omega)} \right)^{2\beta},
		\end{align*}
		where $\eta=\frac{s_{2}}{s_{2}-2}$ with $2<s_{2}<\frac{2\beta q}{\beta q-q+2}$. This concludes step two.
		
		Following from steps one and two, we conclude that	
		\begin{align*}
			\Vert H_{1}(u_{1}) \Vert_{L^{2\chi}(\Omega)}^{2}
			&\leq C_{1} \Vert\overline{b}_1\Vert_{L^{q/2}(\Omega)} \Vert u_{1} H_{1}'(u_{1}) \Vert_{L^{q^{\ast}}(\Omega)}^{2} \\
			&\quad + \beta_{1}^{2}  \frac{2\beta_{1} - 1}{2\beta_{1}} \,
			\Vert \lambda \rho \Vert_{L^{q/2}(\Omega)} \, \bigg(
			\varepsilon_{5} \Vert u_{1}^{\beta_{1}} \Vert_{L^{2\chi}(\Omega)} \\
			&\qquad + C \varepsilon_{5}^{-\sigma_{1}} \Vert u_{1}^{\beta_{1}} \Vert_{L^{q^{\ast}}(\Omega)}
			\bigg)^{2} \\
			&\quad + \frac{\beta_{1}^{2}}{2\beta_{1}} \,
			\Vert \lambda \rho \Vert_{L^{q/2}(\Omega)} \, \bigg(
			\varepsilon_{6} \Vert u_{2}^{\alpha \beta_{1}} \Vert_{L^{2\chi}(\Omega)} 
			+ C \varepsilon_{6}^{-\sigma_{1}} \Vert u_{2}^{\alpha \beta_{1}} \Vert_{L^{q^{\ast}}(\Omega)}
			\bigg)^{2}
		\end{align*}
		and
	\begin{align*}
		\Vert H_{2}(u_{2}) \Vert_{L^{2\eta}(\Omega)}^{2}
		&\leq C_{2} \Vert\overline{b}_2\Vert_{L^{q/2}(\Omega)}\Vert u_{2} H_{2}'(u_{2}) \Vert_{L^{q^{\ast}}(\Omega)}^{2} \\
		&\quad + \beta_{1}^{2}  \frac{2\beta_{1} - 1}{2\beta_{1}} \,
		\Vert \mu \tau \Vert_{L^{q/2}(\Omega)} \, \bigg(
		\varepsilon_{7} \Vert u_{2}^{\beta_{1}} \Vert_{L^{2\eta}(\Omega)} \\
		&\qquad + C \varepsilon_{7}^{-\sigma_{2}} \Vert u_{2}^{\beta_{1}} \Vert_{L^{q^{\ast}}(\Omega)}
		\bigg)^{2} \\
		&\quad + \frac{\beta_{1}^{2}}{2\beta_{1}} \,
		\Vert \mu \tau \Vert_{L^{q/2}(\Omega)} \, \bigg(
		\varepsilon_{8} \Vert u_{1}^{\beta_{1}} \Vert_{L^{2\eta}(\Omega)} 
		+ C \varepsilon_{8}^{-\sigma_{2}} \Vert u_{1}^{\beta_{1}} \Vert_{L^{q^{\ast}}(\Omega)}
		\bigg)^{2\beta}.
	\end{align*}
	
Now, applying the Minkowski inequality and taking into account the definitions of the functions \( u_t \) for \( t = 1,2 \), we obtain
\begin{align}
	\|u_1^{\beta_1}\|_{L^{2\chi q^{\ast}}(\Omega)} 
	&\leq \|u_1^{\beta_1} - k_1^{\beta_1}\|_{L^{2\chi q^{\ast}}(\Omega)} 
	+ |\Omega|^{\frac{1}{2\chi q^{\ast}} - \frac{1}{q^{\ast}}} \|u_1^{\beta_1}\|_{L^{q^{\ast}}(\Omega)} \notag \\
	\|u_2^{\beta_2}\|_{L^{2\eta q^{\ast}}(\Omega)} 
	&\leq \|u_2^{\beta_2} - k_2^{\beta_2}\|_{L^{2\eta q^{\ast}}(\Omega)} 
	+ |\Omega|^{\frac{1}{2\eta q^{\ast}} - \frac{1}{q^{\ast}}} \|u_2^{\beta_2}\|_{L^{q^{\ast}}(\Omega)}.
	\label{inequalitiesH_1andH_2}
\end{align}

	Thus, by combining the definition of \( H_t \), the inequalities \eqref{inequalitiesH_1andH_2}, the identity \( \alpha \beta = 1 \) with \( \alpha \leq 1 \), and the fact that \( \beta_1 \geq 1 \), we have
		\begin{align*}
			\Vert u_{1} \Vert_{L^{2\chi \beta_{1}}(\Omega)}^{2\beta_{1}} 
			&\leq C \beta_{1}^{2} \Vert u_{1} \Vert_{L^{q^{\ast} \beta_{1}}(\Omega)}^{\beta_{1}} 
			+ C \beta_{1}^{2} \left( \varepsilon_{5} \Vert u_{1} \Vert_{L^{2\chi \beta_{1}}(\Omega)}^{\beta_{1}} 
			+ C \varepsilon_{5}^{-\sigma_{1}} \Vert u_{1} \Vert_{L^{q^{\ast} \beta_{1}}(\Omega)}^{\beta_{1}} \right)^{2} \\
			&\quad + C \beta_{1}^{2} \left( \varepsilon_{6} \Vert u_{2}^{\alpha} \Vert_{L^{2\chi \beta_{1}}(\Omega)}^{\beta_{1}} 
			+ C \varepsilon_{6}^{-\sigma_{1}} \Vert u_{2}^{\alpha} \Vert_{L^{q^{\ast} \beta_{1}}(\Omega)}^{\beta_{1}} \right)^{2}
		\end{align*}
		and		
		\begin{align*}
			\Vert u_{2} \Vert_{L^{2\eta \beta_{1}}(\Omega)}^{\alpha \beta_{1}} 
			&\leq C \beta_{1}^{\alpha} \Vert u_{2} \Vert_{L^{q^{\ast} \beta_{1}}(\Omega)}^{\alpha \beta_{1}} 
			+ C \beta_{1}^{\alpha} \left( \varepsilon_{7}^{\alpha} \Vert u_{2} \Vert_{L^{2\chi \beta_{1}}(\Omega)}^{\alpha \beta_{1}} 
			+ C \varepsilon_{7}^{-\alpha\sigma_{2}} \Vert u_{2} \Vert_{L^{q^{\ast} \beta_{1}}(\Omega)}^{\alpha \beta_{1}} \right) \\
			&\quad + C \beta_{1}^{\alpha} \left( \varepsilon_{8} \Vert u_{1} \Vert_{L^{2\chi \beta_{1}}(\Omega)}^{\beta_{1}} 
			+ C \varepsilon_{8}^{-\sigma_{2}} \Vert u_{1} \Vert_{L^{q^{\ast} \beta_{1}}(\Omega)}^{\beta_{1}} \right).
		\end{align*}
		Hence, by applying Hölder's inequality together with the previous inequalities, and using the assumptions \( \beta_1 \geq 1 \) and \( \alpha \leq 1 \), we obtain
		\begin{align*}
			\Vert u_{1} \Vert_{L^{2\chi \beta_{1}}(\Omega)} 
			&\leq C^{1/\beta_{1}} \beta_{1}^{1/\beta_{1}} 
			\Vert u_{1} \Vert_{L^{q^{\ast} \beta_{1}}(\Omega)} 
			+ C^{1/\beta_{1}} \beta_{1}^{1/\beta_{1}} \, \bigg(
			\varepsilon_{5} \Vert u_{1} \Vert_{L^{2\chi \beta_{1}}(\Omega)} \\
			&\quad + C \varepsilon_{5}^{-\sigma_{1}} \Vert u_{1} \Vert_{L^{q^{\ast} \beta_{1}}(\Omega)}
			\bigg) \\
			&\quad + C^{1/\beta_{1}} \beta_{1}^{1/\beta_{1}} \, \bigg(
			\varepsilon_{6} \Vert u_{2}^{\alpha} \Vert_{L^{2\chi \beta_{1}}(\Omega)} 
			+ C \varepsilon_{6}^{-\sigma_{1}} \Vert u_{2}^{\alpha} \Vert_{L^{q^{\ast} \beta_{1}}(\Omega)}
			\bigg)\\
			&\leq C^{1/\beta_{1}} \beta_{1}^{1/\beta_{1}} 
			\Vert u_{1} \Vert_{L^{q^{\ast} \beta_{1}}(\Omega)} 
			+ C^{1/\beta_{1}} \beta_{1}^{1/\beta_{1}} \, \bigg(
			\varepsilon_{5} \Vert u_{1} \Vert_{L^{2\chi \beta_{1}}(\Omega)} \\
			&\quad + C \varepsilon_{5}^{-\sigma_{1}} \Vert u_{1} \Vert_{L^{q^{\ast} \beta_{1}}(\Omega)}
			\bigg) \\
			&\quad + C^{1/\beta_{1}} \beta_{1}^{1/\beta_{1}} \, \bigg(
			\varepsilon_{6}\Vert u_{2} \Vert_{L^{2\chi \beta_{1}}(\Omega)}^{\alpha} 
			+ C \varepsilon_{6}^{-\sigma_{1}} \Vert u_{2} \Vert_{L^{q^{\ast} \beta_{1}}(\Omega)}^{\alpha}
			\bigg)
		\end{align*}
		and
		\begin{align*}
			\Vert u_{2} \Vert_{L^{2\eta \beta_{1}}(\Omega)}^{\alpha} 
			&\leq C^{1/\beta_{1}} \beta_{1}^{\alpha/\beta_{1}} 
			\Vert u_{2} \Vert_{L^{q^{\ast} \beta_{1}}(\Omega)}^{\alpha} 
			+ C^{1/\beta_{1}} \beta_{1}^{\alpha/\beta_{1}} \, \bigg(
			\varepsilon_{7}^{\alpha} \Vert u_{2} \Vert_{L^{2\chi \beta_{1}}(\Omega)}^{\alpha} \\
			&\quad + C \varepsilon_{7}^{-\alpha\sigma_{2}} \Vert u_{2} \Vert_{L^{q^{\ast} \beta_{1}}(\Omega)}^{\alpha}
			\bigg) \\
			&\quad + C^{1/\beta_{1}} \beta_{1}^{\alpha/\beta_{1}} \, \bigg(
			\varepsilon_{8} \Vert u_{1} \Vert_{L^{2\chi \beta_{1}}(\Omega)} 
			+ C \varepsilon_{8}^{-\sigma_{2}} \Vert u_{1} \Vert_{L^{q^{\ast} \beta_{1}}(\Omega)}
			\bigg).
		\end{align*}
		
	Taking \( \eta = \chi \), that is, \( s := s_1 = s_2 \), and using that \( C^{1/\beta_1} \beta_1^{\beta_1} \) is uniformly bounded for \( \beta_1 \geq 1 \), it is then possible to choose \( \varepsilon_j \) (for \( j = 5,6,7,8 \)) small enough independently of \( \beta_1 \), so that we obtain
	\begin{align*}
		\Vert u_{1} \Vert_{L^{2\chi \beta_{1}}(\Omega)} 
		+ \Vert u_{2} \Vert_{L^{2\chi \beta_{1}}(\Omega)}^{\alpha}
		\leq C^{1/\beta_{1}} \beta_{1}^{1/\beta_{1}} 
		\left( \Vert u_{1} \Vert_{L^{q^{\ast} \beta_{1}}(\Omega)} 
		+ \Vert u_{2} \Vert_{L^{q^{\ast} \beta_{1}}(\Omega)}^{\alpha} \right),
	\end{align*}
	where the constant \( C \) depends only on \( |\Omega| \), \( q \), \(\alpha\), \( \|\lambda \rho\|_{L^{q/2}(\Omega)} \), \( \|\mu \tau\|_{L^{q/2}(\Omega)} \), and \( \nu_{t} \) for \( t = 1,2 \).
	
		Hence, \begin{align}\label{2.8}
			\Vert u_{1} \Vert_{L^{\gamma q^{\ast} \beta_{1}}(\Omega)} 
			+ \Vert u_{2} \Vert_{L^{\gamma q^{\ast} \beta_{1}}(\Omega)}^{\alpha} 
			\leq C^{1/\beta_{1}} \beta_{1}^{1/\beta_{1}} 
			\left( \Vert u_{1} \Vert_{L^{q^{\ast} \beta_{1}}(\Omega)} 
			+ \Vert u_{2} \Vert_{L^{q^{\ast} \beta_{1}}(\Omega)}^{\alpha} \right),
		\end{align}
		where \( \gamma = \frac{s(q - 2)}{q(s - 2)} \), which satisfies \( \gamma > 1 \) due to the assumption \( 2 < s < q \).

		Now, by iterating the estimate \eqref{2.8} with \( \beta_1 = \gamma^m \) for each \( m \in \mathbb{N}\cup\{0\} \), we obtain
	\begin{align*}
		\Vert u_{1} \Vert_{L^{\gamma^{m} q^{\ast}}(\Omega)} 
		+ \Vert u_{2} \Vert_{L^{\gamma^{m} q^{\ast}}(\Omega)}^{\alpha}
		&\leq \prod_{j=0}^{m-1} \left(C \gamma^{j} \right)^{\gamma^{-j}} 
		\left[ \Vert u_{1} \Vert_{L^{q^{\ast}}(\Omega)} + \Vert u_{2} \Vert_{L^{q^{\ast}}(\Omega)}^{\alpha} \right] \\
		&\leq C^{n_{1}} \gamma^{n_{2}} 
		\left[ \Vert u_{1} \Vert_{L^{q^{\ast}}(\Omega)} + \Vert u_{2} \Vert_{L^{q^{\ast}}(\Omega)}^{\alpha} \right] \\
		&\leq C 
		\left[ \Vert u_{1} \Vert_{L^{q^{\ast}}(\Omega)} + \Vert u_{2} \Vert_{L^{q^{\ast}}(\Omega)}^{\alpha} \right],
	\end{align*}
		where $n_{1}:=\displaystyle\sum_{j=0}^{m-1}\gamma^{-j}$ and $n_{2}:=\displaystyle\sum_{j=0}^{m-1}j\gamma^{-j}$.
		
		Thus, taking the limit as \( m \to +\infty \), we obtain
		 \begin{align}\label{2.10}
			\Vert u_{1} \Vert_{L^{\infty}(\Omega)} 
			+ \Vert u_{2} \Vert_{L^{\infty}(\Omega)}^{\alpha} 
			\leq C \left[ \Vert u_{1} \Vert_{L^{q^{\ast}}(\Omega)} + \Vert u_{2} \Vert_{L^{q^{\ast}}(\Omega)}^{\alpha} \right].
		\end{align}
		
		Now, using interpolation inequality and \eqref{2.10} we have 
		\begin{align*}
			\Vert u_{1} \Vert_{L^{\infty}(\Omega)} 
			+ \Vert u_{2} \Vert_{L^{\infty}(\Omega)}^{\alpha} 
			\leq C \left[ \Vert u_{1} \Vert_{L^{2}(\Omega)} + \Vert u_{2} \Vert_{L^{2}(\Omega)}^{\alpha} \right].
		\end{align*}
				
		Therefore, combining the previous estimate with the Minkowski inequality, we obtain \begin{align*}
			\Vert u^{+} \Vert_{L^{\infty}(\Omega)} 
			+ \Vert v^{+} \Vert_{L^{\infty}(\Omega)}^{\alpha} 
			\leq C \left[ 
			\Vert u^{+} \Vert_{L^{2}(\Omega)} 
			+ \Vert v^{+} \Vert_{L^{2}(\Omega)}^{\alpha} 
			+ k_1 + k_2^{\alpha} 
			\right],
		\end{align*}
		where $C = C(\nu_t, \lambda, \mu, \Vert \rho \Vert_{L^{q/2}(\Omega)}, \Vert \tau \Vert_{L^{q/2}(\Omega)}, \alpha, q, \vert\Omega\vert)$.
		
	\end{proof}

\noindent\section*{Acknowledgments}
\addcontentsline{toc}{section}{Acknowledgments}

\noindent E.J.F. Leite has been supported by CNPq-Grant 316526/2021-5. The first author gratefully acknowledges the invitation of Prof. Edir Leite to visit UFSCar, and thanks the department for its warm hospitality. He also thanks the Department of Mathematics at UFRR for granting leave to carry out his postdoctoral research.

\vspace{0.5em}

\textbf{Data availability} \quad Data sharing not applicable to this article as no datasets were generated or analyzed during the current study.

\section*{Declarations}

\textbf{Conflict of interest} \quad The authors declare that they have no conflict of interest.

\vspace{0.5em}

\textbf{Ethical approval} \quad The research does not involve humans and/or animals. The authors declare that there are no ethics issues to be approved or disclosed.

\bibliographystyle{abbrv}
\bibliography{AFR}
	
\end{document}